\documentclass[12pt]{amsart}
\usepackage{amsthm}
\usepackage{amssymb}

\usepackage{geometry}
\theoremstyle{plain}
\newtheorem{thm}{Theorem}
  \theoremstyle{plain}
  \newtheorem{prop}[thm]{Proposition}
  \theoremstyle{plain}
  \newtheorem{cor}[thm]{Corollary}
  \theoremstyle{remark}
  \newtheorem{rem}[thm]{Remark}

\subjclass[2010]{16P10, 16D20, 16D25}

\begin{document}

\title[Wedderburn-Malcev decomposition ]{Wedderburn-Malcev decomposition of one-sided ideals of finite dimensional
algebras}

\author{A.A. Baranov}

\thanks{Supported by University of Leicester}

\address{Department of Mathematics, University of Leicester, Leicester, LE1
7RH, UK}

\email{ab155@le.ac.uk}

\author{A. Mudrov}

\thanks{Supported by University of Leicester}

\address{Department of Mathematics, University of Leicester, Leicester, LE1
7RH, UK}

\email{am405@le.ac.uk}

\author{H.M. Shlaka}

\thanks{Supported by the Higher Committee for Educational Development in
Iraq (HCED Iraq). }

\address{Department of Mathematics, University of Leicester, Leicester, LE1
7RH, UK; Department of Material Engineering, University of Kufa, Al-Najaf,
Iraq}

\email{hmass2@le.ac.uk, hasan.shlaka@uokufa.edu.iq}
\begin{abstract}
Let $A$ be a finite dimensional associative algebra over a perfect
field and let $R$ be the radical of $A$. We show that for every
one-sided ideal $I$ of $A$ there exists a semisimple subalgebra
$S$ of $A$ such that $I=I_{S}\oplus I_{R}$ where $I_{S}=I\cap S$
and $I_{R}=I\cap R$. 
\end{abstract}

\maketitle
Wedderburn-Malcev Principal Theorem is one of the landmarks in the
theory of associative algebras.
\begin{thm}[Wedderburn, Malcev]
 Let $A$ be a finite dimensional associative algebra and let $R$
be the radical of $A$. Suppose that $A/R$ is separable. Then the
following hold.

(1) There exists a semisimple subalgebra $S$ of $A$ such that $A=S\oplus R$
(Wedderburn). 

(2) If $Q$ is a semisimple subalgebra of $A$ then there exists $r\in R$
such that $Q\subseteq(1+r)S(1+r)^{-1}$ (Malcev).

(3) If $S_{1}$ and $S_{2}$ are two subalgebras of $A$ such that
$A=S_{i}\oplus R$ ($i=1,2$) then there exists $r\in R$ such that
$S_{1}=(1+r)S_{2}(1+r)^{-1}$ (Malcev). 
\end{thm}
In this note we show that for every one-sided ideal $I$ of $A$ there
exists a Wedderburn-Malcev decomposition $S\oplus R$ of $A$ which
\emph{splits} $I$, i.e. $I=I_{S}\oplus I_{R}$, where $I_{S}=I\cap S$
and $I_{R}=I\cap R$. 

Note that the algebra $A$ in Wedderburn-Malcev theorem is not required
to be unital. If $A$ doesn't contain $1$ then the conjugation $s\mapsto(1+r)s(1+r)^{-1}$
should be rewritten in the obvious way: by expanding $(1+r)^{-1}=1-r+r^{2}-r^{3}+\dots$
(since $r$ is nilpotent the sum is finite) we get $(1+r)s(1+r)^{-1}=s+rs-sr-rsr+\dots$.
Although most of the books state this theorem for unital algebras
only (e.g. \cite{CR}), one can find it in a non-unital context as
well (as in the original papers of Wedderburn and Malcev), see, for
example, \cite{B}. The part (2) of the theorem is not normally mentioned
in the books so we refer the reader to the original Malcev's paper
for the proof \cite{M}. Note that if $A$ is a finite dimensional
algebra over a perfect field then $A/R$ is always separable, so the
theorem holds in this case. 

For the sake of greater generality some of the results below are stated
for the Artinian rings. Although every finite dimensional unital algebra
is Artinian as a ring, this is not true for non-unital algebras (for
example, any one dimensional algebra over $\mathbb{Q}$ with zero
multiplication is not Artinian as a ring). However, as usual, it is
easy to see that the results below hold for the finite dimensional
algebras as well. 

Throughout this paper, $A$ is a (non-unital) associative left Artinian
ring or a finite dimensional algebra. We denote by $R$ its radical.
If $V$ is a subset of $A$ we denote by $\bar{V}$ its image in $\bar{A}=A/R$.
Let $I$ be a left ideal of $A$ and let $Q$ be a left ideal of $\bar{A}$.
We say that $I$ is $Q$-\emph{minimal} (or simply \emph{bar-minimal})
if $\bar{I}=Q$ and for every left ideal $J$ of $A$ with $J\subseteq I$
and $\bar{J}=Q$ one has $J=I$. 

It is well known that every non-nilpotent left ideal of a left Artinian
ring contains an idempotent, see for example \cite[Theorem 24.2]{CR}.
We need a bit more precise version of this fact. 
\begin{prop}
\label{thm:I=00003DAe} Let $A$ be a left Artinian associative ring
and let $I$ be a left ideal of $A$. Suppose that $I$ is bar-minimal.
Then there is an idempotent $e\in I$ such that $I=Ae$. 
\end{prop}
\begin{proof}
We can assume that $\bar{I}$ is non-zero (otherwise $I=0$ is generated
by the idempotent $0$). Since $\bar{A}$ is semisimple, there is
a non-zero idempotent $f\in\bar{I}$ such that $\bar{I}=\bar{A}f$.
Fix any $x\in I$ such that $\bar{x}=f$. Note that $x^{2}\neq0$
because $\bar{x}^{2}=f^{2}=f\neq0$. Consider the map $\varphi_{x}:I\rightarrow I$
defined by $\varphi_{x}\left(y\right)=yx$ for all $y\in I$. Then
$\varphi_{x}$ is an endomorphism of the left $A$-module $I$. Note
that the image $J=\varphi_{x}(I)=Ix$ is a left ideal of $A$. We
have 
\[
\bar{J}=\overline{Ix}=\bar{I}\bar{x}=(\bar{A}f)f=\bar{A}f=\bar{I}.
\]
Since $J\subseteq I$ and $I$ is $\bar{I}$-minimal, we get $I=J=Ix$.
Thus, $\varphi_{x}$ is surjective. Since $I$ is an Artinian module,
$\varphi_{x}$ is an automorphism. Hence, there is an element $e\in I$
such that $x=\varphi_{x}\left(e\right)=ex$. We have 
\[
\varphi_{x}\left(e^{2}\right)=e^{2}x=e\left(ex\right)=ex=\varphi_{x}\left(e\right),
\]
so 
\[
\varphi_{x}\left(e^{2}-e\right)=\varphi_{x}\left(e^{2}\right)-\varphi_{x}\left(e\right)=0
\]
Since $\varphi_{x}$ is injective, $e^{2}-e=0$, so $e$ is an idempotent.
As $\bar{e}\in\bar{I}=\bar{A}f$, we have $\bar{e}=\bar{a}f$ for
some $\bar{a}\in\bar{A}$. Thus, $\bar{e}f=\bar{a}ff=\bar{a}f=\bar{e}$.
On the other hand, we have $ex=x$, so $\bar{e}f=\bar{e}\bar{x}=\bar{x}=f$.
Hence $\bar{e}=f=\bar{x}$. Therefore, 
\[
\overline{Ae}=\bar{A}\bar{e}=\bar{A}f=\bar{I}.
\]
 Since $Ae\subseteq I$ is a left ideal of $A$ with $\overline{Ae}=\bar{I}$
and $I$ is $\bar{I}$-minimal, we must have $I=Ae$, as required. 
\end{proof}
As a corollary we get the following simple fact, which is frequently
mentioned in the textbooks. 
\begin{cor}
\label{cor:I=00003DAe in ass(A)} Let $A$ be a left Artinian ring
and let $I$ be a minimal non-nilpotent left ideal of $A$. Then $I=Ae$
for some idempotent $e\in I$. 
\end{cor}
\begin{proof}
Since $I$ is non-nilpotent, $I\not\subseteq R$, so $\bar{I}$ is
a non-zero left ideal of $\bar{A}$. Let $J$ be a left ideal of $A$
such that $J\subseteq I$ and $\bar{J}=\bar{I}$. Then $J$ is non-nilpotent
(because $\bar{J}=\bar{I}$ is non-nilpotent). Since $I$ is minimal
non-nilpotent, $J=I$. This implies that $I$ is $\bar{I}$-minimal.
Therefore, by Proposition \ref{thm:I=00003DAe}, there is an idempotent
$e\in I$ such that $I=Ae$. 
\end{proof}
The following theorem gives a complete characterisation of bar-minimal
left ideals of left Artinian associative rings. 
\begin{thm}
Let $A$ be a left Artinian associative ring and let $I$ be a left
ideal of $A$. Then $I$ is bar-minimal if and only if there is an
idempotent $e\in A$ such that $I=Ae$. 
\end{thm}
\begin{proof}
The ``only if'' part is proved in Proposition \ref{thm:I=00003DAe}.
Suppose now that $I=Ae$, where $e$ is an idempotent in $A$. Let
$J\subseteq I$ be an $\bar{I}$-minimal left ideal of $A$. We need
to show that $J=I$. By Theorem \ref{thm:I=00003DAe}, $J=Ae_{1}$
for some idempotent $e_{1}$ in $A$. Note that 
\[
e_{1}=e_{1}e_{1}\in Ae_{1}=J\subseteq I=Ae,
\]
so $e_{1}e=e_{1}$. Let $e_{2}=ee_{1}e=ee_{1}$. Then $e_{2}$ is
an idempotent. Indeed,
\[
e_{2}^{2}=ee_{1}eee_{1}e=ee_{1}e_{1}e=ee_{1}e=e_{2}.
\]

We have $Ae_{2}=Aee_{1}\subseteq Ae_{1}\subseteq Ae$. We claim that
$Ae_{2}=Ae_{1}=J$. Since $J$ is $\bar{I}$-minimal, it is enough
to show that $\bar{A}\bar{e}_{2}=\bar{A}\bar{e}=\bar{I}$. As $\bar{A}\bar{e}_{1}=\bar{I}=\bar{A}\bar{e}$,
we have $\bar{e}=\bar{e}\bar{e}\in\bar{A}\bar{e}_{1}$, so $\bar{e}\bar{e}_{1}=\bar{e}$.
Therefore, $\bar{A}\bar{e}_{2}=\bar{A}\bar{e}\bar{e}_{1}=\bar{A}\bar{e}$,
as required.

We are going to show that $e=e_{2}$. Since $\bar{e}=\bar{e}\bar{e}\in\bar{A}\bar{e}=\bar{A}\bar{e}_{2}$,
we have $\bar{e}\bar{e}_{2}=\bar{e}$. Recall that $e_{2}=ee_{1}$,
so 
\[
\bar{e}=\bar{e}\bar{e}_{2}=\bar{e}\bar{e}\bar{e}_{1}=\bar{e}\bar{e}_{1}=\bar{e}_{2}.
\]
Therefore, there is $r\in R$ such that $e_{2}=e+r$. Since $e_{2}=ee_{1}e$,
we have $e_{2}=ee_{2}=e_{2}e$. Hence
\[
e+r=e_{2}=ee_{2}=e(e+r)=e+er,
\]
so $er=r$. Similarly, $re=r$. Thus, 
\[
e+r=e_{2}=e_{2}^{2}=(e+r)^{2}=e+2r+r^{2}.
\]
Therefore, $r^{2}=-r$ and $r^{2^{k}}=-r$ for all $k\in\mathbb{N}$.
Since $R$ is nilpotent, we get $r=0$, so $e_{2}=e$. Therefore,
$I=Ae=Ae_{2}=Ae_{1}=J$, as required.  
\end{proof}
\begin{rem}
Let $e_{1}$ and $e_{2}$ be idempotents of $A$. They are said to
be \emph{right equivalent} if $e_{1}e_{2}=e_{1}$ and $e_{2}e_{1}=e_{2}$.
It is easy to see that $Ae_{1}=Ae_{2}$ if and only if $e_{1}$ and
$e_{2}$ are right equivalent. Thus, the bar-minimal left ideals of
$A$ are in bijective correspondence with the right equivalence classes
of the idempotents of $A$. 
\end{rem}
We are now ready to prove our main result. 
\begin{thm}
Let $A$ be a finite dimensional algebra and let $I$ be a left ideal
of $A$. Suppose that $A/R$ is separable. Then there exists a semisimple
subalgebra $S$ of $A$ such that $A=S\oplus R$ and $I=I_{S}\oplus I_{R}$,
where $I_{S}=I\cap S$ and $I_{R}=I\cap R$. 
\end{thm}
\begin{proof}
If $I$ is nilpotent, then $I\subseteq R$, so $I=I_{R}$ as required.
Suppose that $I$ is non-nilpotent. Then $\bar{I}$ is a non-zero
left ideal of $\bar{A}$. Let $J$ be a minimal left ideal of $A$
such that $J\subseteq I$ and $\bar{J}=\bar{I}$. Then by Proposition
\ref{thm:I=00003DAe}, $J=Ae$ for some idempotent $e\in J$. By Wedderburn-Malcev
Theorem, there is a semisimple subalgebra $S$ of $A$ such that $e\in S$
and $A=S\oplus R$. We have 
\[
J=Ae=\left(S\oplus R\right)e=Se\oplus Re=J_{S}\oplus J_{R}
\]
where $J_{S}=Se$ and $J_{R}=Re$. Note that $J_{S}=J\cap S$ (because
$e\in S$) and $J_{R}=J\cap R$. Put $I_{S}=I\cap S$ and $I_{R}=I\cap R$.
Then the sum $I_{S}+I_{R}\subseteq I$ is direct. We have $J_{S}\subseteq I_{S}$
and 
\[
\bar{J}_{S}\subseteq\bar{I}_{S}\subseteq\bar{I}=\bar{J}=\bar{J}_{S}.
\]
This implies $I/I_{R}\cong\bar{I}=\bar{I}_{S}\cong I_{S}$, so $I=I_{S}\oplus I_{R}$,
as required. 
\end{proof}

\end{document}